\documentclass[12pt,twoside,reqno]{amsart}
\usepackage[latin1]{inputenc}
\usepackage[italian,english]{babel}
\usepackage{amssymb,latexsym}
\usepackage{amsmath,amsfonts,amsthm}
\usepackage{paralist}
\usepackage{color}

\numberwithin{equation}{section}

\newcommand{\R}{\mathbb{R}}
\newcommand{\RR}{\mathbb{R}}
\newcommand{\N}{\mathbb{N}}

\newcommand{\Z}{\mathbb{Z}}

\newcommand{\s}{\sharp}

\DeclareMathOperator{\dive}{div}

\newtheorem{lem}{Lemma}

\newtheorem{thm}{Theorem}

\newtheorem{ex}{Example}
\newtheorem{defn}{Definition}
\theoremstyle{remark}

\begin{document}
\title[Periodic fractional equations]{Periodic solutions for nonlocal\\ fractional equations}
\author{Vincenzo Ambrosio}
\address{Dipartimento di Matematica e Applicazioni ``R. Caccioppoli''\\
         Universit\`{a} degli Studi di Napoli Federico II\\
         via Cinthia, 80126 Napoli, Italy}
\email{vincenzo.ambrosio2@unina.it}
\author{Giovanni Molica Bisci}
\address{Dipartimento P.A.U., Universit\`a  degli
Studi Mediterranea di Reggio Calabria, Salita Melissari - Feo di
Vito, 89100 Reggio Calabria, Italy}
\email{gmolica@unirc.it}

\maketitle
\begin{abstract}
The purpose of this paper is to study the existence of (weak) periodic solutions for
nonlocal fractional equations with periodic boundary conditions.
These equations have a
variational structure and, by applying a critical point result coming out from a
classical Pucci-Serrin theorem in addition to a local minimum result for differentiable
functionals due to Ricceri, we are able to prove the existence of at least two periodic solutions
for the treated problems. As far as we know, all these results are new.
\end{abstract}

\section{Introduction}

\noindent

One of the most celebrated applications of the Mountain Pass
Theorem
consists in the construction
of non-trivial solutions of semilinear equations (see \cite{pucci,pucciradu} and references therein). In \cite{svmountain}, exploiting these Mountain Pass techniques and motivated by the great attention devoted to partial differential equations involving fractional operators (see, for instance, the papers \cite{AP, KMS2, pu2}),
the authors studied the existence of one non-trivial weak solution for the
following nonlocal problem
$$ \left\{
\begin{array}{ll}
(-\Delta)^s u=f(x,u) & {\mbox{ in }} \Omega\\
u=0 & {\mbox{ in }} \RR^N\setminus \Omega,
\end{array} \right.
$$
\noindent where $s\in (0,1)$ is fixed, $N>2s$, $\Omega\subset \RR^N$ is an open bounded set with
continuous boundary, $f:\Omega\times \RR\rightarrow \RR$ is a suitable Carath\'{e}odory function, and $(-\Delta)^s$ is the fractional Laplace operator,
which (up to normalization factors) may be defined as
\begin{equation} \label{2}
-(-\Delta)^s u(x)=
\int_{\RR^n}\frac{u(x+y)+u(x-y)-2u(x)}{|y|^{n+2s}}\,dy,
\,\,\,\,\, (x\in \RR^N).\end{equation}
See also \cite{svlinking} for related topics.\par
\indent In their framework, the weak solutions of problem \eqref{2} are constructed with
a variational method by a minimax procedure on the associated
energy functional.\par
To make the nonlinear methods work, they assume the standard Ambrosetti-Rabinowitz condition
\begin{equation}\label{AR0}
\begin{aligned}
&\mbox{\it there exist}\,\, \mu>2\,\, \mbox{\it and}\,\,r>0\,\, \mbox{\it such that}\,\, \mbox{\it a.e.}\,\, x\in \Omega, t\in \R,\, |t|\geq r\\
& \qquad \qquad \qquad \qquad 0<\mu F(x,t)\le tf(x,t),
\end{aligned}
\end{equation}
where the function $F$ is
the primitive of $f$ with respect to the second variable, that is
\begin{equation*}\label{F}
{\displaystyle F(x,t)=\int_0^t f(x,\tau)d\tau}.
\end{equation*}
More precisely, the classical geometry of the Mountain Pass Theorem
is respected by the nonlocal framework assuming that
\begin{equation}\label{cond0}
{\displaystyle \lim_{|t|\to 0}\frac{f(x,t)}{|t|}=0}\,\, \mbox{\it uniformly in}\,\, x\in \Omega,
\end{equation}
in addition to \eqref{AR0}.\par

Now, the literature on nonlocal operators and on their applications is, therefore,
very interesting and large (see, e.g., \cite{MRS}
for an elementary introduction to this topic and a for a -- still not exhaustive --
list of related references). In this spirit, in the recent papers \cite{A2, A3}, the existence of a non-trivial solution for fractional nonlocal problems
 under periodic conditions has been proved. Also in this context a crucial role is played by \eqref{AR0} and the sublinear behaviour of the nonlinear term $f$ at zero.\par

\indent  Along this direction, in this paper we are concerned with the multiplicity of $T$-periodic (weak) solutions to the following nonlocal problem
\begin{equation}\label{P}
\left\{
\begin{array}{ll}
[(-\Delta+m^{2})^{s}-\gamma]u=\lambda f(x,u) &\mbox{ in } (0,T)^{N}   \\
u(x+Te_{i})=u(x)    &\mbox{ for all } x \in \R^{N}, \quad i=1, \dots, N
\end{array},
\right.
\end{equation}
where $s\in (0,1)$, $N \geq 2$, $m>0$, $0<\gamma<m^{2s}$, $\{e_{i}\}_{i=1}^{N}$ is the canonical basis in $\R^{N}$, $f:\R^N\times \R\rightarrow \R$ is a continuous function, and $\lambda$ is positive real parameter.\par
\indent The main novelty here, respect to the approach considered in \cite{A2, A3, svmountain, svlinking}, is to avoid any condition on the nonlinear term $f$ at zero. We also emphasize that, on the contrary of the classical literature dedicated to periodic boundary value problems involving the Laplace operator or some of its generalizations, up to now, to our knowledge, just a few numbers of papers \cite{Dab, KNV, RS1, RS2} consider periodic nonlocal fractional equations.\par

In order to define the nonlocal operator $(-\Delta+m^{2})^{s}$ we proceed as follows:
let $u\in \mathcal{C}^{\infty}_{T}(\R^{N})$, that is $u$ is infinitely differentiable in $\R^{N}$ and $T$-periodic in each variable. Then $u\in \mathcal{C}^{\infty}_{T}(\R^{N})$ can be represented via Fourier series expansion:
$$
u(x)=\sum_{k\in \Z^{N}} c_{k} \frac{e^{\imath \omega k\cdot x}}{{\sqrt{T^{N}}}} \quad (x\in \R^{N})
$$
where
$$
 \omega:=\frac{2\pi}{T}\mbox{ and } \; c_{k}:=\frac{1}{\sqrt{T^{N}}} \int_{(0,T)^{N}} u(x)e^{- \imath \omega k \cdot x}dx \quad (k\in \Z^{N})
$$
are the Fourier coefficients of the smooth and $T$-periodic function $u$.\par
With the above notations the nonlocal operator $(-\Delta+m^{2})^{s}$ is given by
\begin{equation*}\label{nfrls}
(-\Delta+m^{2})^{s} \,u:=\sum_{k\in \Z^{N}} c_{k} (\omega^{2}|k|^{2}+m^{2})^{s} \, \frac{e^{\imath \omega k\cdot x}}{{\sqrt{T^{N}}}}.
\end{equation*}
\noindent
Furthermore, if $\displaystyle{u:=\sum_{k\in \Z^{N}} c_{k} \frac{e^{\imath \omega k\cdot x}}{{\sqrt{T^{N}}}}}$ and $\displaystyle{v:=\sum_{k\in \Z^{N}} d_{k} \frac{e^{\imath \omega k\cdot x}}{{\sqrt{T^{N}}}}}$, we have that the quadratic form
$$
\mathcal{Q}(u,v):=\sum_{k\in \Z^{N}} (\omega^{2}|k|^{2}+m^{2})^{s} c_{k} \bar{d}_{k}
$$
can be extended by density on the Hilbert space
$$
\mathbb{H}^{s}_{T}:=\left\{u=\sum_{k\in \Z^{N}} c_{k} \frac{e^{\imath \omega k\cdot x}}{{\sqrt{T^{N}}}}\in L^{2}(0,T)^{N}: \sum_{k\in \Z^{N}} (\omega^{2}|k|^{2}+m^{2})^{s} \, |c_{k}|^{2}<+\infty \right\}
$$
endowed with the norm
$$
|u|_{\mathbb{H}^{s}_{T}}:=\left(\sum_{k\in \Z^{N}} (\omega^{2}|k|^{2}+m^{2})^{s} |c_{k}|^{2}\right)^{1/2}.
$$

\indent
Let us recall that in $\R^{N}$, the operator $(-\Delta+m^{2})^{s}$ is strictly related to the quantum mechanics; in fact, when $s= {1}/{2}$, $(-\Delta+m^{2})^{{1}/{2}}-m$ corresponds to the Hamiltonian of a free relativistic particle of mass $m$; see \cite{LL}.
There is also a deep connection between $(-\Delta+m^{2})^{s}-m^{2s}$ and the Stochastic Processes theory; the operator in question is an infinitesimal generator of a L\'{e}vy process $\{X^{m}_{t}\}_{t \geq 0}$ called the relativistic $2s$-stable process
\begin{equation*}\label{carfun}
\mathbb{E}(e^{i \xi \cdot X^{m}_{t}}) := e^{-t [(m^{2}+|\xi|^{2})^{s}-m^{2s}]}  \quad (\xi \in \R^{N});
\end{equation*}
for details we refer to \cite{CMS} and \cite{ryznar}.
\medskip

Finally, for our purpose, we assume that the right-hand side of
equation~\eqref{P} is a continuous function
$f:\R^{N+1}\rightarrow \R$ verifying the following hypotheses:
\smallskip
\begin{compactenum}[($f_1$)]
\item $f(x,t)$ \textit{is }$T$-\textit{periodic in} $x \in \R^{N}$, \textit{that is} $f(x+Te_{i},t)=f(x,t)$ \textit{for any} $x\in \R^{N}$, $T\in \R$, \textit{and }$i=1, \dots, N$;
\item \textit{there exist} $a_{1}, a_{2}>0$, and $\displaystyle{2<q<2^{\s}_{s}:=\frac{2N}{N-2s}}$ \textit{such that}
$$
|f(x,t)|\leq a_{1}+a_{2}|t|^{q-1}
$$
\textit{for any} $x\in \R^{N}$ \textit{and }$t\in \R$;
\item \textit{there exist} $\alpha>2$ and $r_{0}>0$ \textit{such that}
$$
0<\alpha F(x,t)\leq t f(x,t)
$$
\textit{for} $x\in \R^{N}$ \textit{and} $|t|\geq r_{0}$, \textit{where}
 $\displaystyle{F(x,t):=\int_{0}^{t} f(x,\tau) d\tau}$.
\end{compactenum}

\indent It is worth remembering that in some recent papers the authors studied fractional boundary value problems, relaxing hypothesis (\ref{cond0}) and trying to relax or to drop condition $(f_2)$; see, among others, the manuscripts  \cite{BMB0,BMB,BMSNonlinearity,molicaservadei} and the references therein, as well as \cite{CM,MS,SZ} for some nice results obtained in the classical setting.\par
 In this spirit, the main result of the present paper is a multiplicity theorem as stated here below.

\begin{thm}\label{thm1}
Let $m>0$ and $0<\gamma<m^{2s}$. Let $f:\R^{N+1} \rightarrow \R$ be a continuous function satisfying the assumptions $(f_1)$--$(f_3)$.\\
 \indent Then, for any $\varrho>0$ and each
\begin{align}\label{la}
0<\lambda<\frac{\displaystyle{q \sqrt{\varrho}\left(1-\frac{\gamma}{m^{2s}}\right)^{q/2}}}{\displaystyle{2\kappa_{s}\left(a_{1}\sigma_{1}q\left(1-\frac{\gamma}{m^{2s}}\right)^{\frac{q-1}{2}}+a_{2}\sigma_{q}^{q}\varrho^{\frac{q-1}{2}}\right)}},
\end{align}
problem \eqref{P} admits at least two weak solutions in $\mathbb{H}^{s}_{T}$, one of which lies in
$$
\mathbb{S}_{\varrho}:=\left\{u\in \mathbb{H}^{s}_{T}: |u|_{\mathbb{H}^{s}_{T}}<\sqrt{\frac{\varrho}{\kappa_{s} \left(1-\frac{\gamma}{m^{2s}}\right)}}\right\},
$$
where
\begin{align}\label{la2}
\displaystyle{\kappa_{s}:= 2^{1-2s} \frac{\Gamma(1-s)}{\Gamma(s)}},
\end{align}
and
$$
\sigma_{r}:=\frac{1}{\sqrt{\kappa_{s}}}\sup_{u \in \mathbb{H}^{s}_{T}\setminus \{0\}} \frac{|u|_{L^{r}(0, T)^{N}}}{|u|_{\mathbb{H}^{s}_{T}}} \mbox{ with } r\in \{1, q\}.
$$
\end{thm}

\indent
As an application of Theorem \ref{thm1} we give the following example.
\begin{ex}\rm{
Let $N=2$, $s\in (1/2,1)$, and $f(x, t):=1+t^3$. Then, $f$ verifies $(f_1)$, $(f_2)$ with $q=4$, $a_1=a_2=1$, and $(f_3)$ with $\alpha=3$ and $r_{0}=2$. Clearly $f$ does not satisfy (\ref{cond0}).
By applying Theorem \ref{thm1}, we deduce that the following fractional nonlocal problem
\begin{equation*}
\left\{
\begin{array}{ll}
[(-\Delta+m^{2})^{s}-\gamma]u=\lambda(1+u^3) &\mbox{ in } (0,T)^{2}   \\
u(x+Te_{i})=u(x)    &\mbox{ for all }  x\in \R^{2}, i=1, 2
\end{array},
\right.
\end{equation*}
admits at least two distinct and (non-trivial) $T$-periodic weak solutions, whenever $\lambda \in \Lambda$,
where $$\Lambda:=\left(0, \displaystyle\frac{2}{\kappa_s}\left(1-\frac{\gamma}{m^{2s}}\right)^{2}\max _{\varrho>0} h(\varrho)\right)$$ and $h: [0, +\infty) \rightarrow [0, +\infty)$ is the continuous function given by
$$
h(\varrho):=\frac{\displaystyle{\sqrt{\varrho}}}{\displaystyle{4\sigma_{1}\left(1-\frac{\gamma}{m^{2s}}\right)^{\frac{3}{2}}+\sigma_{4}^{4} \varrho^{\frac{3}{2}}}},
$$
for every $\varrho>0$.}
\end{ex}

\indent
Let us point out that on the contrary of \cite{A2, A3} and \cite{svmountain}, in this paper we assume that $f$ satisfies the Ambrosetti-Rabinowitz condition without requiring the usual additional asymptotic assumption (\ref{cond0}) at zero. Then, Theorem \ref{thm1} improves the existence results obtained in \cite{A2, A3} and it can be viewed as a subelliptic counterpart of \cite[Theorem 4]{R0}; see also \cite{MF} for related topics.\par

\indent
The existence result above stated is a consequence of two critical points theorems (a variant of the Mountain
Pass Theorem due to Pucci and Serrin (see \cite{puse}) and a local minimum result for differentiable functionals (see \cite{R2}) due to Ricceri) to
elliptic partial differential equations (see Theorem \ref{PucciSerrinRicceri} below).\par
 More precisely, in the present paper we prove that the geometry
of these classical minimax theorems is respected by the nonlocal framework: for this we
develop a functional analytical setting in order to correctly encode the periodic boundary datum in the variational
formulation.\par
 Of course, also the compactness property required by these minimax
theorems is satisfied in the nonlocal setting, again thanks to the choice of the
functional setting we work in (see Lemmas \ref{lemma1} and \ref{lemma2}).

\indent
Hence, in order to prove our result, we recall \cite{A2, A3} that the nonlocal operator $(-\Delta+m^{2})^{s}$, exactly as for the fractional Laplacian, can be related to a Dirichlet to Neumann operator (see also \cite{CafSil, CDDS, FF, ST,SV} for more closely related models).\par
Thus, instead of (\ref{P}), we investigate the following problem
\begin{equation}\label{R}
\left\{
\begin{array}{ll}
-\dive(y^{1-2s} \nabla v)+m^{2}y^{1-2s}v =0 &\mbox{ in }\mathcal{S}_{T}:=(0,T)^{N} \times (0,\infty)  \\
\smallskip
v_{| {\{x_{i}=0\}}}= v_{| {\{x_{i}=T\}}} & \mbox{ on } \partial_{L}\mathcal{S}_{T}:=\partial (0,T)^{N} \times [0,\infty) \\
\smallskip
{\partial_{\nu}^{1-2s} v}=\kappa_{s} [\gamma v+\lambda f(x,v)]   &\mbox{ on }\partial^{0}\mathcal{S}_{T}:=(0,T)^{N} \times \{0\}
\end{array}
\right.
\end{equation}
where
$$
{\partial_{\nu}^{1-2s} v}:=-\lim_{y \rightarrow 0^{+}} y^{1-2s} \frac{\partial v}{\partial y}(x,y)
$$
is the conormal exterior derivative of $v$.\par
\indent Taking into account variational structure of $(\ref{R})$, its solutions are obtained as critical points of the energy functional $\mathcal{J}_{\lambda}$ given by
 $$
 \mathcal{J}_{\lambda}(v):=\frac{1}{2\lambda}\left( \|v\|_{\mathbb{X}_{T}^{s}}^{2}-\gamma \kappa_{s} |\textup{Tr}(v)|_{L^{2}(0,T)^{N}}^{2}\right) -\kappa_{s}\int_{\partial^{0}\mathcal{S}_{T}} F(x,\textup{Tr}(v)) \,dx,
 $$
defined on the space $\mathbb{X}_{T}^{s}$, which is the closure of the set of smooth and $T$-periodic (in $x$) functions in $\R^{N+1}_{+}$ with respect to the norm
$$
\|v\|_{\mathbb{X}_{T}^{s}}:=\left(\iint_{\mathcal{S}_{T}} y^{1-2s} (|\nabla v|^{2}+m^{2s} v^{2}) dx dy\right)^{1/2},
$$
and $\textup{Tr}(v):=v(\cdot,0)$.\par
\indent We will prove that $\mathcal{J}_{\lambda}$ satisfies the assumptions of an abstract variational principle \cite{R2}
which allow us to deduce that (\ref{P}) has at least two weak solutions in $\mathbb{H}^{s}_{T}$ one of which lies in a ball.


The paper is organized as follows: in Section 2 we recollect some properties of the
about involved functional spaces, and establish a trace inequality corresponding
to this operator; the nonlocal fractional operator $(-\Delta+m^{2})^{s}$ is considered in
Section 2, by means of the use of the harmonic extension; this includes studying
an associated linear equation in the local version. The main section, Section 3, contains
the results related to the nonlocal nonlinear problem (\ref{P}), where we prove
the main result.

We refer to the recent book \cite{MRS}, as well as \cite{DPV}, for the abstract variational setting used along the present paper.

\section{Some Preliminaries}
\indent
In this section we collect preliminary tools related to (\ref{P}) and the reformulated version (\ref{R}), as well as some basic propositions on trace inequality and embedding results firstly proved in \cite{A2} and \cite{A3}. The reader familiar with this topic may skip this section and go directly to the next one. From now on, we consider $s\in (0, 1)$ and we assume $N\geq 2$. Set
$$
\R^{N+1}_{+}=\{(x,y)\in \R^{N+1}: x\in \R^{N}, y>0 \}
$$
be the upper half-space in $\R^{N+1}$.

\indent
We denote by $\mathcal{S}_{T}:=(0,T)^{N}\times(0,\infty)$ the half-cylinder in $\R^{N+1}_{+}$ with basis $\partial^{0}\mathcal{S}_{T}:=(0,T)^{N}\times \{0\}$
and lateral boundary $\partial_{L}\mathcal{S}_{T}:=\partial (0,T)^{N}\times [0,+\infty)$. Further, with $\|v\|_{L^{r}(\mathcal{S}_{T})}$ we always denote the norm of $v\in L^{r}(\mathcal{S}_{T})$ and with $|u|_{L^{r}(0,T)^{N}}$ the $L^{r}(0,T)^{N}$ norm of $u \in L^{r}(0,T)^{N}$.\par
\indent Let $\mathcal{C}^{\infty}_{T}(\R^{N})$ be the space of functions
$u\in \mathcal{C}^{\infty}(\R^{N})$ such that $u$ is $T$-periodic in each variable, that is
$$
u(x+Te_{i})=u(x) \mbox{ for all } x\in \R^{N}, i=1, \dots, N.
$$
 \indent We define the fractional Sobolev space $\mathbb{H}^{s}_{T}$ as the closure of
$\mathcal{C}^{\infty}_{T}(\R^{N})$ endowed by the norm
\begin{equation*}\label{h12norm}
|u|_{\mathbb{H}^{s}_{T}}:=\sqrt{ \sum_{k\in \Z^{N}} (\omega^{2}|k|^{2}+m^{2})^{s} \, |c_{k}|^{2}}.
\end{equation*}
Finally we introduce the functional space $\mathbb{X}^{s}_{T}$ defined as the completion of
\begin{align*}
\mathcal{C}_{T}^{\infty}(\overline{\R^{N+1}_{+}}):=\Bigl\{&v\in \mathcal{C}^{\infty}(\overline{\R^{N+1}_{+}}): v(x+Te_{i},y)=v(x,y) \\
&\mbox{ for every } (x,y)\in \overline{\R_{+}^{N+1}}, i=1, \dots, N \Bigr\}
\end{align*}
under the $H^{1}(\mathcal{S}_{T},y^{1-2s})$ norm
\begin{equation*}
\|v\|_{\mathbb{X}^{s}_{T}}:=\sqrt{\iint_{\mathcal{S}_{T}} y^{1-2s} (|\nabla v|^{2}+m^{2}v^{2}) \, dxdy} \,.
\end{equation*}

In our framework, a fundamental r\^{o}le between the spaces $\mathbb{X}^{s}_{T}$ and $\mathbb{H}_{T}^{s}$ is played by trace operator $\textup{Tr}$. More precisely, one has the following result (see \cite{A2, A3} for details).
\begin{thm}\label{tracethm}
There exists a surjective linear operator $\textup{Tr} : \mathbb{X}^{s}_{T} \rightarrow \mathbb{H}_{T}^{s}$  such that:
\begin{itemize}
\item[$(i)$] $\textup{Tr}(v)=v|_{\partial^{0} \mathcal{S}_{T}}$ for all $v\in \mathcal{C}_{T}^{\infty}(\overline{\R^{N+1}_{+}}) \cap \mathbb{X}^{s}_{T}$;
\item[$(ii)$] $\textup{Tr}$ is bounded and
\begin{equation}\label{tracein}
\sqrt{\kappa_{s}} |\textup{Tr}(v)|_{\mathbb{H}^{s}_{T}}\leq \|v\|_{\mathbb{X}^{s}_{T}},
\end{equation}
for every $v\in \mathbb{X}^{s}_{T}.$
In particular, equality holds in \eqref{tracein} for some $v\in \mathbb{X}^{s}_{T}$ if and only if $v$ weakly solves the following equation
$$
-\dive(y^{1-2s} \nabla v)+m^{2}y^{1-2s}v =0 \, \mbox{ in } \, \mathcal{S}_{T}.
$$
\end{itemize}
\end{thm}

\indent The next embedding results will be crucial in the sequel.
\begin{thm}\label{compacttracethm}
Let $N> 2s$. Then  $\textup{Tr}(\mathbb{X}^{s}_{T})$ is continuously embedded in $L^{q}(0,T)^{N}$ for any  $1\leq q \leq 2^{\s}_{s}$.  Moreover,  $\textup{Tr}(\mathbb{X}^{s}_{T})$ is compactly embedded in $L^{q}(0,T)^{N}$  for any  $1\leq q < 2^{\s}_{s}$.
\end{thm}

Let $g \in \mathbb{H}^{-s}_{T}$, where
$$\mathbb{H}^{-s}_{T}:=\left\{g=\sum_{k\in \Z^{N}} g_{k} \frac{e^{\imath \omega k\cdot x}}{{\sqrt{T^{N}}}}: \sum_{k\in \Z^{N}}\frac{|g_{k}|^{2}}{(\omega^{2} |k|^{2}+m^{2})^{s}}<+\infty\right\}$$ is the dual of $\mathbb{H}^{s}_{T},$, and consider the following two problems:
\begin{equation}\label{P*}
\left\{
\begin{array}{ll}
(-\Delta+m^{2})^{s}u=g &\mbox{ in } (0,T)^{N}  \\
u(x+Te_{i})=u(x) & \mbox{ for all } x\in \R^{N}
\end{array}
\right.
\end{equation}
and
\begin{equation}\label{P**}
\left\{
\begin{array}{ll}
-\dive(y^{1-2s} \nabla v)+m^{2}y^{1-2s}v =0 &\mbox{ in }\mathcal{S}_{T}  \\
v_{| {\{x_{i}=0\}}}= v_{| {\{x_{i}=T\}}} & \mbox{ on } \partial_{L}\mathcal{S}_{T} \\
{\partial_{\nu}^{1-2s} v}=g(x)  &\mbox{ on } \partial^{0}\mathcal{S}_{T}
\end{array}.
\right.
\end{equation}

\begin{defn}
We say that $v \in \mathbb{X}^{s}_{T}$ is a weak solution to \eqref{P**}
if for every $\varphi \in \mathbb{X}^{s}_{T}$ it holds
$$
\iint_{\mathcal{S}_{T}} y^{1-2s} (\nabla v \nabla \varphi + m^{2} v \varphi  ) \, dxdy=\kappa_{s}\langle g, \textup{Tr}(\varphi)\rangle.
$$
Here $\langle \cdot, \cdot \rangle$ is the duality pairing between $\mathbb{H}^{s}_{T}$ and its dual $\mathbb{H}^{-s}_{T}$.
\end{defn}
We have the notion of \textit{weak solution} to problem \eqref{P**}.
\begin{defn}
We say that $u\in \mathbb{H}^{s}_{T}$ is a weak solution to \eqref{P*} if $u=\textup{Tr}(v)$
and $v\in \mathbb{X}^{s}_{T}$ is a weak solution to \eqref{P**}.
\end{defn}

Taking into account Theorem \ref{tracethm} and Theorem \ref{compacttracethm}, it is possible to introduce the notion of extension for a function $u\in \mathbb{H}^{s}_{T}$. More precisely, the next result holds (see also the more general result contained in \cite{ST}).
\begin{thm}
Let $u\in \mathbb{H}^{s}_{T}$. Then, there exists a unique $v\in \mathbb{X}^{s}_{T}$ such that
\begin{equation*}\label{extPu}
\left\{
\begin{array}{ll}
-\dive(y^{1-2s} \nabla v)+m^{2}y^{1-2s}v =0 &\mbox{ in }\mathcal{S}_{T}  \\
v_{| {\{x_{i}=0\}}}= v_{| {\{x_{i}=T\}}} & \mbox{ on } \partial_{L}\mathcal{S}_{T} \\
v(\cdot,0)=u  &\mbox{ on } \partial^{0}\mathcal{S}_{T}
\end{array}
\right.
\end{equation*}
and
\begin{align*}\label{conormal}
-\lim_{y \rightarrow 0^{+}} y^{1-2s}\frac{\partial v}{\partial y}(x,y)=\kappa_{s} (-\Delta+m^{2})^{s}u(x) \mbox{ in } \mathbb{H}^{-s}_{T}.
\end{align*}
We call $v\in \mathbb{X}^{s}_{T}$ the extension of $u\in \mathbb{H}^{s}_{T}$.\par
\indent In particular, if $u=\displaystyle\sum_{k\in \Z^{N}} c_{k} \frac{e^{\imath \omega k\cdot x}}{{\sqrt{T^{N}}}}$, then $v$ is given by
\begin{align*}
v(x,y)=\sum_{k\in \Z^{N}} c_{k} \theta_{k}(y) \frac{e^{\imath \omega k\cdot x}}{{\sqrt{T^{N}}}},
\end{align*}
where $\theta_{k}(y):= \theta(\sqrt{\omega^{2} |k|^{2}+ m^{2}} y)$ and $\theta(y)\in H^{1}(\R_{+},y^{1-2s})$ solves the following ODE
\begin{equation*} \label{ccv}
\left\{
\begin{array}{cc}
&\theta^{''}+\frac{1-2s}{y}\theta^{'}-\theta=0 \mbox{ in } \R_{+}  \\
&\theta(0)=1 \mbox{ and } \theta(\infty)=0
\end{array}.
\right.
\end{equation*}
We note that $$\theta(y)=\displaystyle\frac{2}{\Gamma(s)}\left(\frac{y}{2}\right)K_{s}(y),$$ where $K_{s}$ denotes the modified Bessel function of the second kind with order $s$.\\
\indent Moreover, $v$ satisfies the properties:
\begin{itemize}
\item[$(i)$] $v$ is smooth for $y>0$ and $T$-periodic in $x$$;$
\item[$(ii)$] $\|v\|_{\mathbb{X}^{s}_{T}}\leq \|z\|_{\mathbb{X}^{s}_{T}}$ for any $z\in \mathbb{X}^{s}_{T}$ such that $\textup{Tr}(z)=u$$;$
\item[$(iii)$] $\|v\|_{\mathbb{X}^{s}_{T}}=\sqrt{\kappa_{s}} |u|_{\mathbb{H}^{s}_{T}}$.
\end{itemize}
\end{thm}


\indent
Now, we can reformulate the nonlocal problem (\ref{P}) with periodic boundary conditions, in a local way according the following definitions.
\begin{defn}
Fixing $\lambda>0$, we say that $v \in \mathbb{X}^{s}_{T}$ is a weak solution to \eqref{P**}
if
$$
\iint_{\mathcal{S}_{T}} y^{1-2s} (\nabla v \nabla \varphi + m^{2} v \varphi  ) \, dxdy=\kappa_s\gamma \int_{\partial^{0}\mathcal{S}_{T}} \textup{Tr}(v)\textup{Tr}(\varphi) dx
$$
$$
\qquad+\lambda\kappa_s\int_{\partial^{0}\mathcal{S}_{T}} f(x, \textup{Tr}(v)) \textup{Tr}(\varphi) dx,
$$
for every $\varphi \in \mathbb{X}^{s}_{T}$.
\end{defn}
Finally, we have the notion of \textit{weak solution} to problem \eqref{R} as follows.
\begin{defn}
Fixing $\lambda>0$, we say that $u\in \mathbb{H}^{s}_{T}$ is a weak solution to \eqref{P} if $u=\textup{Tr}(v)$
and $v\in \mathbb{X}^{s}_{T}$ is a weak solution to \eqref{R}.
\end{defn}

\section{Proof of Theorem \ref{thm1}}
\indent
The aim of this section is to prove that under natural assumptions on the nonlinear term, there exist two (weak) solutions to the extended problem \eqref{R}. Let us introduce the following functional $\mathcal{J}_{\lambda}: \mathbb{X}^{s}_{T} \rightarrow \R$ defined by
\begin{align}
\mathcal{J}_{\lambda}(v):=\frac{1}{2 \lambda}\left(\|v\|^{2}_{\mathbb{X}^{s}_{T}} -\gamma \kappa_s |\textup{Tr}(v)|_{L^{2}(0,T)^{N}}^{2}\right)- \kappa_s\int_{\partial^{0}\mathcal{S}_{T}} F(x,\textup{Tr}(v)) dx,
\end{align}
where $\lambda>0$ is fixed.

\indent We will prove that $\mathcal{J}_{\lambda}$ satisfies the assumptions of the following abstract result:
\begin{thm}\label{PucciSerrinRicceri}
Let $E$ be a reflexive real Banach space, and let $\Phi,\Psi:E\to\R$
be two continuously G\^{a}teaux differentiable functionals such that $\Phi$ is
sequentially weakly lower semicontinuous and coercive. Further, assume that $\Psi$
is sequentially weakly continuous.
In addition, assume that, for each $\mu>0$, the functional $J_\mu:=\mu\Phi-\Psi$ satisfies the classical compactness Palais-Smale $($briefly $(\rm PS)$$)$ condition. Then, for each $\varrho>\displaystyle\inf_E\Phi$ and each
$$
\mu>\inf_{u\in\Phi^{-1}\big((-\infty,\varrho)\big)}
\frac{\displaystyle\sup_{v\in\Phi^{-1}\big((-\infty,\varrho)\big)}\Psi(v)-\Psi(u)}{\varrho-\Phi(u)}\,
$$
the following alternative holds$:$ either the functional $J_\mu$ has a strict global minimum which lies in $\Phi^{-1}\big((-\infty,\varrho))$, or $J_\mu$ has at least two critical points one of which lies in $\Phi^{-1}\big((-\infty,\varrho))$.
\end{thm}

\noindent
\indent For a proof of Theorem \ref{PucciSerrinRicceri}, one can see \cite{R0}. We refer also to \cite{MR, Ri1, R3} and references therein for recent applications of Ricceri's variational principle \cite{R2}.\par

\smallskip
\indent
By using $(f_2)$ and Theorem \ref{compacttracethm}, it follows that $\mathcal{J}_{\lambda}$ is well defined, $\mathcal{J}_{\lambda}\in \mathcal{C}^{1}(\mathbb{X}^{s}_{T},\R)$ and
$$
\langle \mathcal{J}'_{\lambda}(v), \varphi\rangle = \frac{1}{\lambda}\left(\iint_{\mathcal{S}_{T}} y^{1-2s} (\nabla v \nabla \varphi+m^{2} v \varphi) dx dy-\gamma \kappa_s\int_{\partial^{0}\mathcal{S}_{T}} \textup{Tr}(v) \textup{Tr}(\varphi) dx\right)
$$
$$
\quad\quad-\kappa_s\int_{\partial^{0}\mathcal{S}_{T}} f(x, \textup{Tr}(v)) \textup{Tr}(\varphi) dx,
$$
for every $\varphi \in \mathbb{X}^{s}_{T}$.

\indent
Let $\varrho>0$ and set $\mu:=1/(2\lambda)$, with $\lambda$ as in the statement of Theorem \ref{thm1}.\\
We set  $E:=\mathbb{X}^{s}_{T}$, $J_\mu:=\mathcal{J}_{\lambda}$ and
$$
\Phi(v):=\|v\|_{\mathbb{X}^{s}_{T}}^2-\gamma\kappa_s|\textup{Tr}(v)|^{2}_{L^{2}(0, T)^{N}}=:\|v\|_{e}^{2},
$$
as well as
$$
\Psi(v):=\displaystyle\kappa_s\int_{\partial^{0} \mathcal{S}_{T}} F(x,\textup{Tr}(v))dx,
$$
for every $v \in \mathbb{X}^{s}_{T}$.\par

Taking into account that for any $u=\displaystyle\sum_{k\in \Z^{N}} c_{k} \frac{e^{\imath \omega k\cdot x}}{{\sqrt{T^{N}}}} \in \mathbb{H}^{s}_{T}$
\begin{align*}
m^{2s}|u|_{L^{2}(0,T)^{N}}^{2}&=m^{2s} \sum_{k\in \Z^{N}} |c_{k}|^{2} \\
&\leq  \sum_{k\in \Z^{N}} (\omega^{2}|k|^{2}+m^{2})^{s} |c_{k}|^{2}= |u|_{\mathbb{H}^{s}_{T}}^{2},
\end{align*}
and by using Theorem \ref{tracethm}, it is easily seen that

\begin{align*}\label{partequad}
\|v\|^{2}_{\mathbb{X}^{s}_{T}} -\gamma \kappa_{s} |\textup{Tr}(v)|_{L^{2}(0,T)^{N}}^{2}&= \|v\|^{2}_{\mathbb{X}^{s}_{T}} - \frac{\gamma}{m^{2s}} \kappa_{s} m^{2s} |\textup{Tr}(v)|_{L^{2}(0,T)^{N}}^{2}\\
&\geq \|v\|^{2}_{\mathbb{X}^{s}_{T}} -\frac{\gamma}{m^{2s}} \kappa_{s} |\textup{Tr}(v)|_{\mathbb{H}^{s}_{T}}^{2}\\
&\geq \left(1- \frac{\gamma}{m^{2s}} \right) \|v\|^{2}_{\mathbb{X}^{s}_{T}}.
\end{align*}
\indent This fact and $0<\gamma<m^{2s}$, yield
\begin{equation}\label{equivalent}
\sqrt{1-\frac{\gamma}{m^{2s}}}\|v\|_{\mathbb{X}^{s}_{T}}\leq \|v\|_{e}\leq \|v\|_{\mathbb{X}^{s}_{T}}
\end{equation}
for every $v \in \mathbb{X}^{s}_{T}$, so $\|\cdot\|_{e}$ is equivalent to the standard norm $\|\cdot\|_{\mathbb{X}^{s}_{T}}$.\par
\indent Then, $\Phi$ is sequentially weakly lower semicontinuous and coercive, and $\Psi$ is sequentially weakly continuous thanks to Theorem \ref{compacttracethm}. Now, we show that there exists $v_{0} \in \mathbb{X}^{s}_{T}$ such that
\begin{equation}\label{inf}
J_\mu(tv_0)\rightarrow -\infty
\end{equation}
as $t\rightarrow+\infty$.

By using the Ambrosetti-Rabinowitz condition, it is easy to see that the potential $F(x,t)$ restricted to $[0, T]^{N}\times\R$ is $\alpha$-superhomogeneous at infinity, i.e.
\begin{equation}\label{suphom}
F(x,tv)\geq F(x,v)t^{\alpha},
\end{equation}
for every $x\in [0, T]^{N}$, $(t,v)\in \R^{2}$ with $t\geq 1$ and $|v|\geq r_0$.\par
\indent Indeed, for $t=1$, the equality is obvious. Otherwise, fix $|v|\geq r_0$ and define $g(x,t):=F(x,tv)$, for every $x\in [0, T]^{N}$ and $t\in (1,+\infty)$. By $(f_3)$ it follows that
$$
\frac{\partial_{t}g(x,t)}{g(x,t)}\geq \frac{\alpha}{t},
$$
for every $x\in [0, T]^{N}$ and $t> 1$. By integrating in $(1,t]$ we get that
$$
\int_{1}^{t}\frac{\partial_{t}g(x,\tau)}{g(x,\tau)} d\tau= \log \frac{g(x,t)}{g(x,1)}\geq \log t^{\alpha}.
$$
As a consequence, one has
\begin{align*}
g(x,t)&=F(x,tv) \geq g(x,1)t^{\alpha}= F(x,v)t^\alpha,
\end{align*}
for every $x\in [0, T]^{N}$, $|v|\geq r_0$ and $t> 1$.

\indent
Then, by using (\ref{suphom}), it follows that
\begin{align*}
J_\mu(tv_0)& =\mu\Phi(tv_0)-\Psi(tv_0)\\
& \leq \mu t^2 \Phi(v_0)-\kappa_{s}t^{\alpha}\int_{\{|\textup{Tr}(v_{0})|\geq r_0\}}F(x,\textup{Tr}(v_{0}))dx+\kappa_{s}\nu_1 T^{N},
\end{align*}
where
$$
\nu_{1}:=\sup\{|F(x,t)|:x\in [0, T]^{N}\,\mbox{and}\, |t|\leq r_0\}.
$$

Since $\alpha>2$, choosing $v_{0} \in \mathbb{X}^{s}_{T}$ such that
$$
|\{x\in (0, T)^{N}:|\textup{Tr}(v_{0})|\geq r_0\}|>0,
$$
we deduce that \eqref{inf} holds.
Hence, the functional $J_\mu$ is unbounded from below. \\
In order to apply Theorem \ref{PucciSerrinRicceri}, we need to prove the compactness Palais-Smale condition for the functional $\mathcal{J}_{\lambda}$.\par
\indent For the sake of completeness, we recall that a $\mathcal{C}^1$-functional $J:E\to\R$, where $E$ is a real Banach space with topological dual $E^{*}$, satisfies the \textit{Palais-Smale condition at level} $\zeta\in\R$, (abbreviated $\textrm{(PS)}_{\zeta}$) when:
\begin{itemize}
\item[$\textrm{(PS)}_{\zeta}$] {\it Every sequence $\{v_j\}_{j\in \N}$ in $E$ such that
$$
J(v_{j})\to \zeta, \quad \mbox{ and }\quad \|J'(v_{j})\|_{E^{*}}\to0,
$$
as $j\rightarrow +\infty$, possesses a convergent subsequence.}
\end{itemize}

\indent
We say that $J$ satisfies the \textit{Palais-Smale condition} (abbreviated $\textrm{(PS)}$) if $\textrm{(PS)}_{\zeta}$ holds for every $\zeta\in \R$.

\smallskip
\indent
Proceeding as in \cite{MF}, in the next two lemmas we verify the compactness $(\rm PS)$ condition for the functional $\mathcal{J}_{\lambda}$:

\begin{lem}\label{lemma1}
Assume that conditions $(f_1),(f_2)$ and $(f_3)$ are satisfied. Then, every Palais-Smale sequence for the functional $\mathcal{J}_{\lambda}$ is bounded in $\mathbb{X}^{s}_{T}$.
\end{lem}
\begin{proof}
Let $\{v_{j}\}_{j\in \N}\subset \mathbb{X}^{s}_{T}$ be a Palais-Smale sequence i.e.
\begin{equation}\label{E1}
\mathcal{J}_\lambda(v_{j})\rightarrow \zeta,
\end{equation}
for $\zeta\in \R$ and
\begin{equation}\label{E2}
\|\mathcal{J}_{\lambda}'(v_j)\|_{\mathbb{X}^{-s}_{T}}\to0,
\end{equation}
as $j\rightarrow +\infty$, where
$$
\|\mathcal{J}_{\lambda}'(v_{j})\|_{\mathbb{X}^{-s}_{T}}:=\sup\Big\{ \big|\langle\,\mathcal{J}'_\lambda(v_{j}),\varphi\,\rangle \big|\,: \;
\varphi\in
\mathbb{X}^{s}_{T}, \|\varphi\|_{\mathbb{X}^{s}_{T}}=1\Big\}.
$$
\indent
Suppose by contradiction that $\{v_{j}\}_{j\in \N}$ is not bounded in $\mathbb{X}^{s}_{T}$.
Then, up to a subsequence, we may assume that
$$
\|v_j\|_{\mathbb{X}^{s}_{T}}\rightarrow +\infty \mbox{ as } j\rightarrow +\infty.
$$
It follows that for every $j\in \N$
 \begin{align}\label{sire}
\mathcal{J}_{\lambda}(v_{j})&- \frac{\langle \mathcal{J}'_{\lambda}(v_{j}),v_{j}\rangle }{\alpha}\nonumber \\
&= \frac{1}{2\lambda}\left(\frac{\alpha-2}{\alpha}\right)[\|v_j\|^2_{\mathbb{X}^{s}_{T}}-\gamma \kappa_{s}|\textup{Tr}(v_{j})|^{2}_{L^{2}(0, T)^{N}}] \nonumber \\
& +\kappa_{s}\int_{\partial^{0}\mathcal{S}_{T}}\left[\frac{f(x,\textup{Tr}(v_{j}))\textup{Tr}(v_{j})}{\alpha}-F(x,\textup{Tr}(v_{j}))\right]dx.
\end{align}
Combining (\ref{sire}) and (\ref{equivalent}) we get, for any $j\in \N$
\begin{align*}
&\frac{1}{2\lambda}\left(\frac{\alpha-2}{\alpha}\right)\Bigl(1-\frac{\gamma}{m^{2s}}\Bigr)\|v_j\|^2_{\mathbb{X}^{s}_{T}}\\
&\leq \mathcal{J}_\lambda(v_j)-\frac{\langle \mathcal{J}'_\lambda(v_j),v_j\rangle }{\alpha} \\
&- \kappa_{s}\int_{|\textup{Tr}(v_{j})|>r_0}\left[\frac{f(x,\textup{Tr}(v_{j}))\textup{Tr}(v_{j})}{\alpha}-F(x,\textup{Tr}(v_{j}))\right]dx + \kappa_{s}\nu_2 T^{N},
\end{align*}
where
$$
\nu_{2}:=\sup\left\{\left|\frac{f(x,t)t}{\alpha}-F(x,t)\right|:x\in [0, T]^{N}, |t|\leq r_0\right\}.
$$
\indent
Now, we observe that condition $(f_3)$ yields
$$
\int_{|\textup{Tr}(v_{j})|>r_0}\left[\frac{f(x,\textup{Tr}(v_{j}))\textup{Tr}(v_{j})}{\alpha}-F(x,\textup{Tr}(v_{j}))\right]d\xi\geq 0,
$$
so, we deduce that for every $j\in \N$
\begin{align*}
\frac{1}{2\lambda}\left(\frac{\alpha-2}{\alpha}\right)\left(1-\frac{\gamma}{m^{2s}}\right) \|v_j\|^2_{\mathbb{X}^{s}_{T}}\leq \mathcal{J}_{\lambda}(v_j)-\frac{\langle \mathcal{J}_{\lambda}'(v_j),v_j\rangle }{\alpha}+\kappa_{s}\nu_2T^{N}.
\end{align*}

\noindent Then, for every $j\in \N$ we have
\begin{align*}
\frac{1}{2\lambda}\left(\frac{\alpha-2}{\alpha}\right)\Bigl(1-\frac{\gamma}{m^{2s}}\Bigr)\|v_j\|^2_{\mathbb{X}^{s}_{T}}\leq \mathcal{J}_{\lambda}(v_j)+\frac{1}{\alpha}{\|\mathcal{J}_{\lambda}'(v_j)\|_{\mathbb{X}^{-s}_{T}}\|v_j\|_{\mathbb{X}^{s}_{T}}}+\kappa_{s}\nu_2 T^{N}.
\end{align*}

Dividing both members by $\|v_j\|_{\mathbb{X}^{s}_{T}}$ and letting $j\rightarrow +\infty$, we obtain a contradiction.
\end{proof}

As an immediate consequence of Lemma \ref{lemma1} we are able to prove the following compactness property.
\begin{lem}\label{lemma2}
Assume that conditions $(f_1),(f_2)$ and $(f_3)$ are verified. Then, the functional $\mathcal{J}_{\lambda}$
satisfies the compactness $(\rm PS)$ condition.
\end{lem}

\begin{proof}

Let $\{v_{j}\}_{j\in \N}\subset \mathbb{X}^{s}_{T}$ be a Palais-Smale
sequence. By Lemma \ref{lemma1}, the sequence $\{v_j\}_{j\in \N}$ is bounded in $\mathbb{X}^{s}_{T}$, so we can extract a subsequence,
which for simplicity we denote again by $\{v_{j}\}_{j\in \N}$, such that
$v_{j}\rightharpoonup v$ in $\mathbb{X}^{s}_{T}$. \par
\indent This means that
\begin{align}\label{convergenze0}
\iint_{\mathcal{S}_{T}} y^{1-2s} (\nabla v_{j} \nabla \varphi+m^{2} v_{j} \varphi) dxdy \rightarrow \iint_{\mathcal{S}_{T}} y^{1-2s} (\nabla v \nabla \varphi+m^{2} v \varphi) dx dy
\end{align}
for any $\varphi\in \mathbb{X}^{s}_{T}$,
as $j\to +\infty$.\par
\indent We will prove that ${v_{j}}$ strongly converges to $v\in \mathbb{X}^{s}_{T}$.
Firstly, we can observe that
\begin{align}\label{jj}
\langle a(v_{j}),v_{j}-v \rangle &= \langle \mathcal{J}_{\lambda}'(v_j),v_j-v \rangle+\frac{\gamma}{\lambda} \int_{\partial^{0} \mathcal{S}_{T}} \textup{Tr}(v_{j})(\textup{Tr}(v_{j})-\textup{Tr}(v)) dx \nonumber \\
&+\int_{\partial^{0} \mathcal{S}_{T}}f(x,\textup{Tr}(v_{j}))(\textup{Tr}(v_{j})-\textup{Tr}(v))dx,
\end{align}
where we set
\begin{align}\label{finestra}
\langle a(v_{j}),v_{j}-v \rangle &:= \frac{1}{\lambda}\Bigg(\|v_j\|_{\mathbb{X}^{s}_{T}}^2
-\iint_{\mathcal{S}_{T}} y^{1-2s} (\nabla v_{j} \nabla v+m^{2} v_{j} v) dx dy \Bigg).
\end{align}

\indent Taking into account $\|\mathcal{J}_{\lambda}'(v_j)\|_{\mathbb{X}^{-s}_{T}}\to 0$ and $\{v_j-v\}_{j\in\N}$ is bounded in $\mathbb{X}^{s}_{T}$,
by
$$
|\langle
\mathcal{J}_{\lambda}'(v_j),v_j-v\rangle|\leq\|\mathcal{J}_{\lambda}'(v_j)\|_{\mathbb{X}^{-s}_{T}}\|v_j-v\|_{\mathbb{X}^{s}_{T}},
$$
follows that
\begin{eqnarray}\label{j2}
\langle \mathcal{J}_{\lambda}'(v_j),v_j-v \rangle\to0
\end{eqnarray}
as $j\rightarrow +\infty$.\par
Since the embedding $\textup{Tr}(\mathbb{X}^{s}_{T})\hookrightarrow L^{q}(0, T)^{N}$ is compact, it is clear that up to a subsequence, $\textup{Tr}(v_{j})\to \textup{Tr}(v)$ strongly in $L^q(0, T)^{N}$.
So, by condition $(f_2)$, we obtain
\begin{eqnarray}\label{j3}
\int_{\partial^{0}\mathcal{S}_{T}}|f(x,\textup{Tr}(v_{j}))||\textup{Tr}(v_{j})-\textup{Tr}(v)|dx\to0.
\end{eqnarray}
\indent Putting together \eqref{jj}, \eqref{j2} and \eqref{j3} we get
\begin{eqnarray}\label{fin}
\langle a(v_{j}),v_{j}-v \rangle
\rightarrow 0,
\end{eqnarray}
when $j\rightarrow +\infty$.
Then, in view of \eqref{finestra}, \eqref{fin} reads
\begin{eqnarray}\label{fin2}
\|v_{j}\|^{2}_{\mathbb{X}^{s}_{T}}-\iint_{\mathcal{S}_{T}} y^{1-2s} (\nabla v_{j} \nabla v+m^{2} v_{j} v) dx dy\rightarrow 0,
\end{eqnarray}
when $j\rightarrow +\infty$.
By \eqref{fin2} and \eqref{convergenze0} it follows that
\begin{align*}
\limsup_{j\rightarrow \infty}\iint_{\mathcal{S}_{T}} y^{1-2s} (\nabla v_{j} \nabla v+m^{2} v_{j} v) dx dy=\iint_{\mathcal{S}_{T}} y^{1-2s} (|\nabla v|^{2}+m^{2} v^{2}) dx dy.
\end{align*}
\indent Since $\mathbb{X}^{s}_{T}$ is a Hilbert space, we can infer that
$$
\|v_{j}-v\|_{\mathbb{X}^{s}_{T}}^{2}=\|v_{j}\|_{\mathbb{X}^{s}_{T}}^{2}+\|v\|_{\mathbb{X}^{s}_{T}}^{2}-2\langle v_{j}, v \rangle_{\mathbb{X}^{s}_{T}} \rightarrow 0,
$$
as $j \rightarrow +\infty$.
\end{proof}

\indent
Now, we are in the position to apply Theorem \ref{PucciSerrinRicceri}. In order to obtain a multiplicity result, we aim to prove that
\begin{equation}\label{piripi}
\mu> \chi(\varrho):=\displaystyle\inf_{u\in \mathbb{B}_\varrho}\frac{\displaystyle \kappa_{s} \sup_{v\in \mathbb{B}_\varrho} \int_{\partial^{0} \mathcal{S}_{T}} F(x,\textup{Tr}(v))dx-\kappa_{s}\int_{\partial^{0} \mathcal{S}_{T}} F(x,\textup{Tr}(u))dx}{\varrho-\|u\|_{e}^{2}}
\end{equation}
for every $\varrho> 0$, where $\mathbb{B}_\varrho:=\{v\in \mathbb{X}^{s}_{T}: \|v\|_{e}<\sqrt{\rho}\}$.
\smallskip

\indent
Fix $\varrho>0$. Since $0\in \mathbb{B}_\varrho$, it follows that
\begin{equation}\label{Funzionale1}
\chi(\varrho)\leq \frac{\displaystyle \kappa_{s} \sup_{v\in \mathbb{B}_\varrho} \int_{\partial^{0} \mathcal{S}_{T}} F(x,\textup{Tr}(v))dx}{\varrho} \, .
\end{equation}
By using $(f_2)$ we can see that
\begin{equation*}
\int_{\partial^{0} \mathcal{S}_{T}} F(x,\textup{Tr}(v))dx \leq a_1 |\textup{Tr}(v)|_{L^1(0, T)^{N}}+\frac{a_2}{q} |\textup{Tr}(v)|^{q}_{L^{q}(0, T)^{N}}
\end{equation*}
for every $v\in \mathbb{X}^{s}_{T}$, so, by (\ref{equivalent}), we deduce that
$$
\sup_{v\in \mathbb{B}_\varrho}\int_{\partial^{0} \mathcal{S}_{T}} F(x,\textup{Tr}(v))dx\leq \displaystyle{{\sigma_1}\sqrt{\varrho}\frac{a_{1}}{\sqrt{1-\displaystyle\frac{\gamma}{m^{2s}}}}+\frac{\sigma_{q}^{q}a_2}{q \left(1-\displaystyle\frac{\gamma}{m^{2s}}\right)^{q/2}}\varrho^{\frac{q}{2}}}.
$$
\indent
This implies that
\begin{equation}\label{Funzionale2}
\frac{\displaystyle{\sup_{v\in \mathbb{B}_\varrho}\int_{\partial^{0} \mathcal{S}_{T}} F(x,\textup{Tr}(v))dx}}{\varrho}\leq \displaystyle{\frac{\sigma_1}{\sqrt{\varrho}}\frac{a_{1}}{\sqrt{1-\displaystyle\frac{\gamma}{m^{2s}}}}+\frac{\sigma_q^qa_2}{q\left(1-\displaystyle\frac{\gamma}{m^{2s}}\right)^{q/2}}\varrho^{\frac{q}{2}-1}} \,.
\end{equation}

\noindent
Since (\ref{la}) holds, conditions (\ref{Funzionale1}) and (\ref{Funzionale2}) immediately yield
$$
\chi(\varrho)\leq \kappa_{s}\left [\displaystyle{\frac{\sigma_1}{\sqrt{\varrho}}\frac{a_{1}}{\sqrt{1-\displaystyle\frac{\gamma}{m^{2s}}}}+\frac{\sigma_q^qa_2}{q\left(1-\displaystyle\frac{\gamma}{m^{2s}}\right)^{q/2}}\varrho^{\frac{q}{2}-1}} \right]<\frac{1}{2\lambda}=:\mu.
$$
\indent
Thus, inequality (\ref{piripi}) is proved.
Then, in view of Theorem \ref{PucciSerrinRicceri}, problem $(\ref{P})$
admits at least two weak solutions one of which lies in $\mathbb{S}_\varrho$. This completes the proof of Theorem \ref{thm1}.\par


\smallskip
\noindent {\bf Acknowledgements.}
The authors warmly thank the anonymous referee for her/his
useful and nice comments on the paper. The manuscript was realized within the auspices of the
INdAM - GNAMPA Projects 2016 titled: {\it Problemi variazionali su variet\`a Riemanniane e gruppi di Carnot}.

\end{document}